\documentclass[letterpaper]{article}
\usepackage[left=3cm,right=3cm,top=3.5cm]{geometry}
\usepackage{indentfirst}
\usepackage{amsmath}
\usepackage{amsthm}
\usepackage{amssymb}
\usepackage{amsfonts}
\usepackage{stmaryrd}
\usepackage{tikz-cd}
\usepackage{tikz}
\usepackage{enumitem}

\def\Gm{\mathbb{G}_m}
\def\SO{\mathrm{SO}}
\def\so{\mathfrak{so}}
\def\Orth{\mathrm{O}}
\def\Sp{\mathrm{Sp}}
\def\sp{\mathfrak{sp}}

\def\GL{\mathrm{GL}}
\def\sl{\mathfrak{sl}}
\def\GLn{{\GL_n}}
\def\GLm{{\GL_m}}
\def\Weil{\mathrm{Weil}}
\def\WeilGH{{\mathrm{Weil}_{G,H}}}
\def\ss{\mathrm{ss}}
\def\diag{\mathrm{diag}}
\def\Mat{\mathrm{Mat}}

\def\Gr{\mathrm{Gr}}

\def\IC{\mathrm{IC}}
\def\dV{{\delta_V}}
\def\dG{{\delta_G}}
\def\dH{{\delta_H}}

\def\Mp{\widetilde{\Sp}}
\def\Gl{{G^\vee}}
\newcommand{\gl}{{\mathfrak{g}^\vee}}
\def\gln{{\mathfrak{gl}_n}}
\def\glm{{\mathfrak{gl}_m}}
\def\GO{{G_{\mathcal{O}}}}

\def\Hl{{H^\vee}}
\def\hl{{\mathfrak{h}^\vee}}
\def\HO{{H_\mathcal{O}}}

\def\VO{{V_\mathcal{O}}}
\def\VF{{V_F}}

\def\wtimes{\widetilde{\times}}

\def\Perv{\mathrm{Perv}}
\def\perf{\mathrm{perf}}

\def\Rep{\mathrm{Rep}}
\def\QCoh{\mathrm{QCoh}}

\def\Sat{\mathrm{Sat}}

\DeclareMathOperator{\Hom}{Hom}
\DeclareMathOperator{\End}{End}

\DeclareMathOperator{\Ext}{Ext}
\DeclareMathOperator{\Sym}{Sym}

\DeclareMathOperator{\SingSupp}{SS}

\DeclareMathOperator{\diagmat}{diag}

\def\transpose{\mathrm{t}}
\def\cartan{\mathfrak{t}}
\def\sym{\mathfrak{S}}
\def\Sn{{\sym_n}}
\def\Sm{{\sym_m}}

\def\bbA{\mathbb{A}}
\def\bbC{\mathbb{C}}
\def\bbZ{\mathbb{Z}}
\def\Fq{\mathbb{F}_q}
\def\calO{{\mathcal{O}}}
\def\calL{\mathcal{L}}
\def\calF{\mathcal{F}}
\def\calS{\mathcal{S}}
\def\Heis{\mathrm{Heis}}
\def\diff{\mathrm{d}}
\DeclareMathOperator{\quo}{\backslash}

\def\Kos{\Sigma}
\def\pt{\mathrm{pt}}
\def\std{\mathrm{std}}
\def\Ad{\mathrm{ad}}
\DeclareMathOperator{\res}{res}

\def\bbP{\mathbb{P}}
\def\Pn{{\bbP^{n-1}}}
\def\Pm{{\bbP^{m-1}}}

\def\LagV{{\widetilde{\calL}(V)}}
\DeclareMathOperator{\Spec}{Spec}

\DeclareMathOperator{\gRes}{gRes}
\DeclareMathOperator{\Res}{Res}
\DeclareMathOperator{\Irr}{Irr}
\def\shear{{\mathbin{\mkern-6mu\fatslash}}}
\newcommand{\sh}[1]{{#1}^\shear}

\newcommand{\gr}[1]{\mathrm{gr}(#1)}
\def\OS{\calO(S)}

\def\stimes{\operatorname*{\times}\limits}
\DeclareMathOperator*{\conv}{\ast}
\def\cn{\!{\conv\limits_\GLn}\!}
\def\cm{\!{\conv\limits_\GLm}\!}
\def\cG{{\conv\limits_G}}
\def\cH{{\conv\limits_H}}
\DeclareMathOperator{\pr}{pr}
\DeclareMathOperator{\act}{act}
\DeclareMathOperator{\proj}{proj}
\DeclareMathOperator{\add}{add}
\DeclareMathOperator{\mult}{mult}
\def\FG{\calF_G}

\makeatletter
\newcommand*\leftdash{\rotatebox[origin=c]{-45}{$\dabar@\dabar@\dabar@$}}
\newcommand*\rightdash{\rotatebox[origin=c]{45}{$\dabar@\dabar@\dabar@$}}
\makeatother

\def\sslash{/\!\!/}
\def\ssslash{/\!\!/\!\!/}

\newtheorem{defn}{Definition}
\newtheorem{lem}{Lemma}
\newtheorem{prop}{Proposition}
\newtheorem{thm}{Theorem}
\newtheorem{conj}{Conjecture}
\newtheorem{cor}{Corollary}
\newtheorem*{remark}{Remark}

\begin{document}
\title{Derived Weil Representation and Relative Langlands Duality}
\author{Haoshuo Fu}
\date{}
\maketitle
\begin{abstract}
	The Weil representation is a particularly significant linear representation of the metaplectic group, used in the study of theta correspondence. In this paper, I introduce a derived category version of the Weil representation in the local field case. For the dual pair $ (\GLn,\GLm) $, I give a coherent description of this category, in the philosophy of relative Langlands duality. 
\end{abstract}
\tableofcontents
\section{Introduction}
\subsection{Weil representations}
The Weil representation is a special representation of symplectic group. The finite field case is defined as follows: let $V$ be a symplectic vector space over the finite field $k=\Fq$ with odd characteristic, and $\Heis(V)$ be the Heisenberg group defined by the symplectic form: 
\begin{equation*}
	1\rightarrow k\rightarrow \Heis(V)\rightarrow V\rightarrow 1. 
\end{equation*}
For a character $\psi\colon \rightarrow \bbC^\times$, we can define an irreducible representation $H_{V,\psi}$ of $\Heis(V)$ with central character $\psi$. It is a subspace of functions on the set $V$ and its dimension is $q^{\frac12\dim V}$. This can be extend to a projective representation $\omega_\psi$ of $\Sp(V)$, called the Weil representation. In general, it can be descent to a representation of the double cover $ \Mp(V) $ of the symplectic group. The case in local field $k((t))$ is defined similarly using residue. 

A dual pair $(G,H)$ is the subgroup $G\times H\rightarrow \Sp(V)$ such that they are the centralizer of each other. Examples are $(\Sp(V_1),\Orth(V_2))$ where $V=V_1\otimes V_2$, and $(\GL(L_1),\GL(L_2))$ where $V=\Hom(L_1,L_2)\oplus\Hom(L_2,L_1)$. By restricting Weil representation to this subgroup, we obtain $\WeilGH$ as a representation of $G\times H$. Associated to it, we can define theta functions and construct theta lifts by using it as an integral kernel. 

By choosing a Lagrangian $ L\subset V $, the Weil representation can be identified with $ L^2 $-functions on $ L $ or $ V/L $. Thus it has a natural categorification $ D(L) $. In \cite{SR07}, the action of $ D(\Sp(V)) $ is constructed via the functor
\begin{equation*}
	D(\Sp(V))\rightarrow \End(D(L))\simeq D(L\times L)\simeq D(V)
\end{equation*}
giving by a sheaf in $ D(\Sp(V)\times V) $. 

In the local field case, one geometric model of Weil representation is constructed in \cite{LL09}. 

When studying Weil representations, we would expect more compatibilities such as the commutativity of these two actions. By mimicking the lattice model of the Weil representation, I could define the derived Weil category with the action of Hecke categories of $ G\times H $ at the same time. 
\begin{thm}
	Let $ F $ be a local field and $ \calO $ its ring of integers. For a variety $ X $, let $ X_\calO $ be its arc space and $ X_F $ be its loop space. 
	
	Let $ \WeilGH $ be the category of $ \GO\times\HO $-equivariant $ (\VO,\psi) $-equivariant sheaves on $ V_F $. Let $ \Sat_G= D_{\GO}(G_F/G_\calO) $ be the derived Satake category. Then we have the action of $ \Sat_{G\times H}\simeq\Sat_G\otimes\Sat_H $ on $ \WeilGH $. Hence the actions of $ \Sat_G $ and $ \Sat_H $ commute in the strongest sense. 
\end{thm}

In \cite{Ly11}, Lysenko constructed the functor from the heart of derived Satake category to the semisimplification of the heart of the Weil category
\begin{equation*}
	\Perv_{\GO}(\Gr_G)\simeq\Rep(\Gl)\rightarrow (\Weil^\heartsuit_{G,H})^\ss, 
\end{equation*}
and showed that this is an equivalence in the case of $ (\GLn,\GLm) $-case and conjectured it is also true in the $ (\Sp_{2m},\SO_{2n}) $-cases. We will show this conjecture is true in section 3. 

Under derived Satake equivalence \cite{BF08}, we can construct the functor 
\begin{equation*}
	D_{\GO}(\Gr_G)\simeq \QCoh_\perf^\Gl(\gl^*[2])\rightarrow \WeilGH.
\end{equation*}
However, this functor is not an isomorphism in general as the left hand side does not have any information of $H$. Hence a natural question is to give a coherent description of $ \WeilGH $ in terms of $ \Gl $ and $ \Hl $. 

Consider the following cases: 
\begin{itemize}[itemsep=0pt]
	\item $ G=\GLn, H=\GLm, n<m $. Let $ (e,h,f) $ be a principal $ \sl_2 $-triple in $ \mathfrak{gl}_{m-n} $ and further embedded into $ \mathfrak{gl}_m=\hl $;
	\item $ G=\SO_{2n}, H=\Sp_{2m}, n\le m $. Let $ (e,h,f) $ be a principal $ \sl_2 $-triple in $ \so_{2m-2n+1} $ and further embedded into $ \so_{2m+1}=\hl $;
	\item $ G=\Sp_{2n}, H=\SO_{2m}, n<m $. Let $ (e,h,f) $ be a principal $ \sl_2 $-triple in $ \so_{2m-2n-1} $ and further embedded into $ \so_{2m}=\hl $. 
\end{itemize}
Let $ S $ be the Slodowy slice $ f+\mathfrak{z}_\hl(e) $ corresponding to the $ \sl_2 $-triple $ (e,h,f) $ inside $ \hl $ and hence inside $ \hl^* $ using the canonical isomorphism $ \hl^*\simeq \hl $. $ S $ carries the action of $ \Gl $ because it acts trivially on the $ \sl_2 $-triple. Besides, $ S $ carries the grading defined by $ t^2\exp(th) $ commuting with $ \Gl $-action. Let $ \sh{S} $ be the dg-scheme with this cohomological grading. 
\begin{conj}
	We have the equivalence of categories
	\begin{equation*}
		\WeilGH \simeq \QCoh_\perf^\Gl(\sh{S}),
	\end{equation*}
	and the equivalence is compatible with the action of derived Hecke categories of $ G $ and $ H $ on both side. 
\end{conj}
\begin{remark}\normalfont
	The action of $ D_{\GO}(\Gr_G)\simeq \QCoh_\perf^\Gl(\gl^*[2]) $ comes from the stack map $ \sh{S}/\Gl \rightarrow \gl^*[2]/\Gl $. 
	%
	For the action of $ \Sat_H\simeq \QCoh_\perf^\Hl(\hl^*[2]) $, it is first mapped to $ \QCoh_\perf^{\Gl\times\Gm}(\hl^*[2]) $, which is equivalent to $ \QCoh_\perf^{\Gl\times\Gm}(\sh{\hl^*}) $. This category acts on $ \QCoh_\perf^\Gl(\sh{S}) $ via the stack map $ \sh{S}/\Gl \rightarrow \hl^*[2]/\Gl $. 
\end{remark}
In this paper, the first case is proved: 
\begin{thm}
	In the case of $ G=\GLn, H=\GLm, n<m $, the categories are equivalent. If the functor in section 5 of \cite{Ly11} is symmetric monoidal, then the above statement about Hecke action is true. 
\end{thm}

For the case $ G=H=\GLn $, the space $ S $ in the equivalence is $ \gl^*\oplus\std\oplus\std^* $ and this result is claimed by Tsao-Hsien Chen and Jonathan Wang. 
\subsection{Relative Langlands duality}
In \cite{GP92}, Gross and Prasad proposed the problem of restricting representations of $\SO_n$ to $\SO_{n-1}$. For irreducible representations $\pi_1$ of $\SO_n$ and $\pi_2$ of $\SO_{n-1}$, to find the multiplicity of trivial representation in $\pi_1\boxtimes\pi_2$ as a representation of $\SO_{n-1}$ requires to calculate the matrix coefficients $\int_{\SO_{n-1}}\langle\pi_1\boxtimes\pi_2(g)v,v^\vee\rangle\diff g$. In \cite{II10}, authors proved that when $v$ is spherical, this is equal to
\begin{equation*}
	\Delta_{\SO_n}\dfrac{L(\frac12,\pi_1\boxtimes\pi_2,\std)}{L(0,\pi_1\boxtimes\pi_2,\Ad)},
\end{equation*}
where $\Delta_{\SO_n}$ is a constant. $\std$ is the standard representation of Langlands dual group of $\SO_{n-1}\times\SO_n$, and $\Ad$ is the adjoint representation. 

Sakellaridis and Venkatesh \cite{SV17} conjectured a generalized result regarding a group $G$ and its spherical variety $X$, which are $\SO_{n-1}\times\SO_n$ and $\SO_{n-1}\quo\SO_{n-1}\times\SO_n$ in the previous discussion. In \cite{Sak13}, Sakellaridis gave the description of $C_c^\infty(X_F)^{\GO}$ under this framework. The categorical version of this conjecture, proposed by Ben-Zvi, Sakellaridis and Venkatesh in \cite{BZSV}, is as follows: the category $D_{\GO}(X_F)$ is equivalent to $\QCoh_\perf^{G_X^\vee}(\sh{V_X})$ for some group $G_X^\vee \rightarrow \Gl$ and its representation $V_X$ with a compatible grading. 

This categorical equivalence for $(G,X)=(\GL_{n-1}\times\GL_n,\GL_{n-1}\quo\GL_{n-1}\times\GL_n)$ is proved in \cite{BFGT}, and the case for $(G,X)=(\SO_{n-1}\times\SO_n,\SO_{n-1}\quo\SO_{n-1}\times\SO_n)$ is proved in \cite{BFT1}. 

For GGP problem of Bessel case, i.e., $ \SO_n $ and $ \SO_m $ when $ m-n $ is odd, the Jacobi group $ J=\SO_n\ltimes U_{m,n}^\SO\subset \SO_n\times\SO_m $ is used. In the case $ G=\SO_n\times\SO_m, X=J\quo G $, it is expected that $ G_X^\vee=\Gl $ and $ V_X $ is the standard representation. The case for $ G=\GLm\times\GLn, X=\GLn\ltimes U_{m,n}^\GL\quo G $ is proved in \cite{TR23}. 

In the framework of \cite{BZSV}, relative Langlands duality is between the pairs $ (G,M=T^*X) $ and $ (\Gl,M^\vee=V_X\times^{G_X^\vee}\Gl) $. Our result verifies $ (G^{\vee\vee}, M^{\vee\vee})=(G,X) $ in the Bessel period case and general linear group case. In fact, one can verify that $ T^*(\SO_m/_fU_{m,n})=\SO_m\times^{U_{m,n}}(f+U_{m,n}^\perp)=\SO_m\times S $. 

Note that the construction of dual space in \cite{BZSV} is not self-dual a priori. For example, it is not clear if $ M^\vee $ is hypersphrical for a general $ M $. 
\subsection{Connection with Coulomb branches}
In \cite{BFN18}, the authors give a mathematical definition of Coulomb branch $\mathcal{M}_{G,N}$ for a group $G$ and its representation $N$ and showed that it only depends on the symplectic representation $T^*N=N\oplus N^*$. 

Recently, \cite{noncot} gives the construction for any symplectic representation. The method is by the geometric Weil representation. For the metaplectic group $\Mp(V)$, consider its $\Mp(V)_\mathcal{O}$-equivariant Weil representation category $\WeilGH$ and the special object $\IC_0$. Then we obtain an algebraic object as inner $\Hom$ of $\IC_0$ in the derived Satake category of $\Mp(V)$. 

In the case of dual pair $ (\GL_n,\GL_m) $ or $(\SO_{2n},\Sp_{2m})$, the anomaly condition in \cite{noncot} is satisfied. The $!$-pullback gives a genuine object in the derived Satake category of $\SO_{2n}\times\Sp_{2m}$. So we can take global section to get an algebra and the Coulomb branch of group $\SO_{2n}\times\Sp_{2m}$ and its representation $\std\otimes \std$. 

Furthermore, the $!$-pullback of an inner $\Hom$ is still an inner $\Hom$, the above construction is exactly considering the inner $\End$ of $\IC_0\in\WeilGH$. From the equivalence of categories $\WeilGH\simeq \QCoh_\perf(\sh{S}/\Gl)$, which identifies $\IC_0$ and the structure sheaf of $\sh{S}/\Gl$. Then the inner $\Hom$ of $\calO_{\sh{S}/\Gl}$ in $\QCoh(\gl^*[2]/G\times\hl^*[2]/H)$ is just the pushforward of $\calO_{\sh{S}/\Gl}$. 

By \cite{BF08}, taking equivariant cohomology as $H^*_{\GO\times\HO}(\pt)$-module refers to the pullback along Kostant section $\Kos_\gl\times\Kos_\hl \rightarrow \gl^*[2]/\Gl\times\hl^*[2]/\Hl$. Hence the coulomb branch in this case is
\begin{equation*}
	S/\Gl\stimes_{\gl^*[2]/\Gl\times\hl^*[2]/\Hl}(\Kos_\gl\times\Kos_\hl). 
\end{equation*}

In \cite{NT17}, the authors showed that the Coulomb branch associated with a quiver of affine type $ A $ with Cherkis bow varieties. If we apply this result to the following quiver: 
\begin{center}
\begin{tikzpicture}
	\draw (0,0) circle (1);
	\fill (30:1) circle (1pt);
	\fill (150:1) circle (1pt);
	\fill (270:1) circle (1pt);
	\node at (1.2,0.6) {$ m $};
	\node at (-1.2,0.6) {$ n $};
	\node at (0,-1.3) {$ 0 $};
\end{tikzpicture}
\end{center}
we get 
\begin{equation*}
	((\GLm\times \Kos_\GLm)\times(\GLm\times S)\times(\GLn\times\Kos_\GLn))\ssslash(\GLm\times\GLn), 
\end{equation*}
where $ \ssslash $ means the Hamiltonian quotient. This is exactly what is stated above in the case of $ (G,H)=(\GLn,\GLm) $. 
\section{Definition of the categories}
\subsection{Notations}
Let $ k $ be an algebraically closed field used in the definition of geometric object. Let $ \Lambda=\overline{\mathbb{Q}_\ell} $ or $ \bbC $ be the field of the coefficient of sheaves. $ \psi\colon k \rightarrow \Lambda^\times $ is a non-trivial character. Then we get the Artin-Schreier sheaf $ \calL_\psi\in D(\mathbb{A}^1) $. In the case $ k=\Lambda=\bbC $, this is the exponential D-module. 

$ F=k((t)) $ is the field of Laurent series, and $ \calO=k[[t]] $ is the ring of integers in $ F $. $ \psi $ naturally extends to a character of $ F $ via residue: $ \psi\colon F\xrightarrow{\res}k\xrightarrow{\psi}\Lambda^\times $. 

When $ V $ is a symplectic vector space, use $ \omega\colon V\times V\rightarrow k $ to denote the symplectic pairing
. It naturally extends to a symplectic pairing on $ V_F $: 
\begin{equation*}
	V_F\times V_F\rightarrow F\xrightarrow{\res}k,
\end{equation*}
which also gives a pairing on $ t^{-r}V_\calO/t^rV_\calO $. By abuse of notation, we still use $ \omega $ to denote them. 

For an algebraic group $ G $, define $ \Gr_G=G_F/G_\calO $ be the affine Grassmannian of $ G $ and define $ \Sat_G=D_{\GO}(\Gr_G) $ be the derived Satake category. Moreover, for positive integer $ n $, let $ G_n $ be the group defined by $ G_n(R)=G(R[[t]]/(t^n)) $. 

If needed, we assume our categories are $ (\infty,1) $-categories. By saying derived category, we mean stable $ (\infty,1) $-categories. For a derived category $ \mathcal{C} $ with certain t-structure, we use $ \mathcal{C}^\heartsuit $ to denote the heart of this t-structure. 
\subsection{Schr\"odinger model}
In the general linear group case, the vector space $ V $ has a polarization $ V=T^*L $ such that $ L=\Hom(V_1,V_2) $ is a representation of $ G\times H $. In this case, the Weil representation can be identified with $ \GO\times \HO $-equivariant sheaves on $ L_F $. More concretely, it is defined as a colimit of categories of the diagram: 
\begin{equation*}
	\cdots\rightarrow D_{G_{2r}\times H_{2r}}(t^{-r}L_\calO/t^rL_\calO)\rightarrow D_{G_{2r+2}\times H_{2r+2}}(t^{-r-1}L_\calO/t^{r+1}L_\calO)\rightarrow \cdots.
\end{equation*}
The arrows are given by $ i_*p^\dagger=i_*p^*[\dim L] $, where $ p\colon t^{-r}L_\calO/t^{r+1}L_\calO \rightarrow t^{-r}L_\calO/t^rL_\calO$ is the projection and $ i\colon t^{-r}L_\calO/t^{r+1}L_\calO \rightarrow t^{-r-1}L_\calO/t^{r+1}L_\calO$ is the inclusion. The degree is chosen such that the middle perverse t-structure is preserved. 
\subsection{Lattice model}
When the case $ V $ is possibly not canonically split, the above construction lacks the equivariance structure. We propose another approach through the so-called lattice model. We first explain our construction through the finite case. 
\subsubsection{Finite case}
Pick any Lagrangian $ L\subset V $, we can think of $ L $ as a group acting on $ V $ via addition. Then we have a relative character on $ L $: $ L\times V\xrightarrow{\psi\circ\omega}\Lambda^\times $ and corresponding sheaf $ \omega^*\calL_\psi $. Call a sheaf $ \calF $ is $ (L,\psi) $-equivariant if we have an isomorphism 
\begin{equation*}
	\mathrm{act}^*\calF\cong \mathrm{proj}^*\calF\otimes\omega^*\calL_\psi. 
\end{equation*}
Hence we can form the category $ D(V/(L,\psi)) $ of $ (L,\psi) $-equivariant sheaves on $ V $. 
\subsubsection{Local case}
Consider the $ \GO\times \HO $-stable Lagrangian $ \VO \subset V_F $. To mimic the finite case, we want a category $ D(V_F/(\VO,\psi)) $. As the colimit of finite cases, we define this category as the colimit of the following diagram: 
\begin{equation*}
	\cdots \rightarrow D((t^{-r}\VO/t^r\VO)/(\VO/t^r\VO,\psi))\xrightarrow{i_*p^\dagger} D((t^{-r-1}\VO/t^{r+1}\VO)/(\VO/t^{r+1}\VO,\psi))\rightarrow \cdots.
\end{equation*}
Even $ G_r $ can act on the space $ t^{-r}\VO/\VO $, it cannot act on $ (\VO/t^r\VO,\psi) $-equivariant sheaves on $ t^{-r}\VO/t^r\VO $. Rather, we only have the action of $ G_{2r} $. Hence the unramified Weil representation $ D_{\GO\times\HO}(V_F/(\VO,\psi)) $ is the colimit of the following diagram: 
\begin{multline*}
	\cdots \rightarrow D_{G_{2r}\times H_{2r}}((t^{-r}\VO/t^r\VO)/(\VO/t^r\VO,\psi))\rightarrow\\
	\rightarrow D_{G_{2r+2}\times H_{2r+2}}((t^{-r-1}\VO/t^{r+1}\VO)/(\VO/t^{r+1}\VO,\psi))\rightarrow \cdots.
\end{multline*}
We will define the Hecke action in the next section. 
\subsection{Fourier transform}
While the lattice model is defined without the assumption of $ V $ having a polarization, we want to show this construction is equivalent to the Schr\"odinger model in polarizable case. 

By the colimit description of the category, it suffices to show $ D(t^{-r}L_\calO/t^rL_\calO) $ is equivalent to $ D((t^{-r}\VO/t^r\VO)/(\VO/t^r\VO,\psi)) $. By taking Fourier transform, we know the latter is equivalent to $ D((t^{-r}\VO/t^r\VO)/(t^{-r}L_\calO/t^rL_\calO,\psi)) $. Hence it suffices to show the following statement: 
\begin{prop}
	If a particular splitting of the short exact sequence $ 0\rightarrow L\rightarrow V\rightarrow V/L\rightarrow 0 $ is chosen, we get a non-canonical equivalence of categories
	\begin{equation*}
		D(V/(L,\psi))\cong D(V/L). 
	\end{equation*}
	If the splitting preserves $ G $-action, we have $ D_G(V/(L,\psi))\cong D_G(V/L).  $
\end{prop}
\begin{proof}
	Consider the space $ L \times V/L $. It carries an $ L $-action by $ L \times L \times V/L \rightarrow L\times V/L $ by $ (l_1,l_2,v+L)\mapsto (l_1+l_2,v+L) $. From the map $ L \times L \times V/L \rightarrow \mathbb{A}^1, (l_1,l_2,v+L)\mapsto \omega(l_1,v) $, we can define $ (L,\psi) $-equivariant sheaves on $ L\times V/L $. 
	
	Then we have the canonical equivalence $ D(V/L)\simeq D((L\times V/L)/L)\simeq D((L\times V/L)/(L,\psi)) $, where the second is given by $ \calF\mapsto \calF\otimes \calL_\psi $. This comes from $ \calL_\psi $ is $ (L,\psi) $-equivariant, as $ \omega(l_1+l_2,v)=\omega(l_1,v)+\omega(l_2,v). $
	
	For a given section $ V/L\rightarrow V$, we get a non-canonical isomorphism $ V \cong L \times V/L $. This isomorphism makes the following diagram commutes: 
	\begin{equation*}
		\begin{tikzcd}
			\mathbb{A}^1 \ar[d,equal]&L\times V\lar\ar[r,shift left,"\act"]\ar[r,shift right,"\proj"']\ar[d,"\cong"]&V\ar[d,"\cong"]\\
			\mathbb{A}^1&L\times L\times V/L\lar\ar[r,shift left,"\act"]\ar[r,shift right,"\proj"']&L\times V/L
		\end{tikzcd}
	\end{equation*}
	This gives the equivalence $ D(V/(L,\psi))\cong D((L\times V/L)/(L,\psi)) $. 
	
	If the $ G $-action preserves the isomorphism $ V\cong L\times V/L $, the above equivalences preserves $ G $-actions. 
\end{proof}
\section{Irreducible objects}
\subsection{Cotangent space}
Here we compute $ T^*(V/(L,\psi)) $. The character $ \omega $ induces a map $ \mathbb{A}^1\times V\rightarrow \mathrm{Lie}(L)^*\simeq L^* $ given by $ V\xrightarrow{\omega} V^* \rightarrow L^* $. The moment map of $ L $-action $ T^*V\rightarrow L^* $ is given by $ (v,v^*)\mapsto (l\mapsto \langle l,v^*\rangle) $. Its fiber at $ 1\in\mathbb{A}^1 $ is 
\begin{equation*}
	T^*V\times_{L^*\times V}({1}\times V)=\{(v,v^*):\omega(v)|_L=v^*|_L\}=\{(v,v^*):v-\omega^{-1}(v^*)\in L\}. 
\end{equation*}
Here the last equation uses the fact that $ L $ is a Lagrangian, i.e., 
\begin{equation*}
	0\rightarrow L\rightarrow V\simeq V^*\rightarrow L^* \rightarrow 0
\end{equation*}
is an exact sequence. Hence we have 
\begin{equation*}
	T^*(V/(L,\psi))=(T^*V\times_{L^*\times V}({1}\times V))/L\simeq V. 
\end{equation*}
Similarly, we should expect $ T^*(V_F/(\VO,\psi))\simeq V_F $. In fact, we see the singular support of sheaves in $ D(V_F/(\VO,\psi)) $ lies in the colimit of the sets
\begin{equation*}
	\cdots\rightarrow\{L\subset t^{-r}\VO/t^r\VO\mbox{ is Lagrangian}\}\xrightarrow{p^*i_*} \{L\subset t^{-r-1}\VO/t^{r+1}\VO\mbox{ is Lagrangian}\}\rightarrow \cdots,
\end{equation*}
which is Lagrangians in $ V_F $ that contains some $ t^N\VO $. 

Then we consider the behavior of $ \GO $-action on sheaves to its singular support. 
\begin{prop}
	The moment map of the $ \GO $-action is given by $ V_F\rightarrow \mathfrak{g}_\calO^*, v\mapsto (g\mapsto\omega(v,gv)) $. 
\end{prop}
\begin{proof}
	First, for the finite case, if a group $ G $ acts on the symplectic space $ (V,\omega) $ and fixes the Lagrangian $ L $, we show the moment map of $ G $-action on $ V/(L,\psi) $ is by $ V\rightarrow \mathfrak{g}^*, v\mapsto(g\mapsto\omega(v,gv)) $. 
	
	The moment map of $ G $-action on $ V $ is by $ T^*V\rightarrow \mathfrak{g}, (v,v^*)\mapsto(gv,v^*) $. It restricts to a map from $ T^*V\times_{L^*\times V}({1}\times V) $. The isomorphism $ T^*V\times_{L^*\times V}({1}\times V)\simeq V $ is given by $ (v,v^*)\mapsto \frac12(v+\omega^{-1}(v^*)) $ or $ v\mapsto \{(v+l,\omega(v-l))\}/L $. Hence the image of $ V $ under the moment map is $ g\mapsto\omega(g(v+l),v-l)=\omega(gv,v). $
	
	Then, for the local case, we have moment maps $ t^{-r}\VO/t^r\VO \rightarrow \mathfrak{g}_{2r}^*,v\mapsto(g\mapsto\omega(v,gv)) $. It is clear they are compatible for different $ r $. By taking colimit, we get the desired moment map $ V_F\rightarrow \mathfrak{g}_\calO^* $. 
\end{proof}
\subsection{Singular support}
The above result is compatible with the singular support calculated using Schr\"odinger models. 

Recall Lysenko's definition of the heart of derived Weil representation in \cite{Ly11} as follows:
\begin{defn}
	Set $ V_i=U_i\oplus U_i^* $. Let $ P_G $ be the parabolic subgroup of $ G $ preserving $ U_1 $, and $ Q_G=\GL(U_1) $ be the standard Levi factor of $ P_G $. Let $ P_H $ be the parabolic subgroup of $ H $ preserving $ U_2 $, and $ Q_H=\GL(U_2) $ be the standard Levi factor of $ P_H $. 
	
	Let $ \Upsilon=\Hom(U_1,V_2) $, and $ \Pi=\Hom(V_1,U_2) $. 
	
	The category of Weil representation $ \Weil^\heartsuit_{G,H} $ is the category of triples $ (\calF_1,\calF_2,\beta) $, where $ \calF_1\in \Perv_{(Q_G\times H)_\calO}(\Upsilon(F)),  \calF_2\in\Perv_{(G\times Q_H)_\calO}(\Pi(F)) $, and $ \beta\colon\zeta(f(\calF_1))\xrightarrow{\sim}f(\calF_2) $ is an isomorphism for the diagram 
	\begin{equation*}
		\begin{tikzcd}
			\Perv_{(Q_G\times H)_\calO}(\Upsilon(F))\dar{f}&\Perv_{(G\times Q_H)_\calO}(\Pi(F))\dar{f}\\
			\Perv_{(Q_G\times Q_H)_\calO}(\Upsilon(F))\rar{\zeta}&\Perv_{(Q_G\times Q_H)_\calO}(\Pi(F))
		\end{tikzcd}
	\end{equation*}
	Where $ f $ are forgetful functors, $ \zeta $ is the partial Fourier transform. 
\end{defn}

The Fourier transform preserves singular support, so the singular support of $ f(\calF_1)\in T^*\Upsilon(F)=V(F) $ is the same as the singular support of $ f(\calF_2)\in T^*\Pi(F)=V(F) $. As $ \calF_1 $ is an element in $ \Perv_{(Q_G\times H)_\calO}(\Upsilon(F)) $, it is $ H_\calO $-equivariant. Thus $ \SingSupp(\calF_1) $ is in the zero fiber of moment map to $ \mathfrak{h}^*_\calO $ and $ H_\calO $-equivariant. Similarly for $ \SingSupp(\calF_2) $. 

In conclusion, we see the singular support of $ (\calF_1,\calF_2,\beta) $ is inside the set 
\begin{equation*}
	\{v\in V(F):v^*v\in\mathfrak{g}_\calO,vv^*\in\mathfrak{h}_\calO\}.
\end{equation*}

\subsection{Relevant orbits}
If a $ (\VO,\psi) $-equivariant sheaf on $ V_F $ is $ \GO $-equivariant, its singular support must be contained in the preimage of $ 0\in\mathfrak{g}_\calO^* $. 

Any section $ V_F/V_\calO \rightarrow V_F $ induces a non-canonical equivalence $ D(V_F/(V_\calO,\psi)) $ with $ D(V_F/V_\calO) $, which does not preserve $ \GO $-action. However, by singular support calculation, we can still determine when a $ \GO $-orbit on $ V_F/V_\calO $ that could occur as the support of an irreducible object in $ D_{\GO}(V_F/(V_\calO,\psi)) $. 

\begin{prop}\label{GL index}
	Let $ V=\Hom(\bbC^n,\bbC^m) $ and $ n\le m $. Consider the subset 
	\begin{equation*}
		\{(v,v^*):v^*v\in\gln(\calO),vv^*\in\glm(\calO)\}\subset V(F)\times V^*(F)
	\end{equation*}
	and its image in $ V(F)/V(\calO)\times V^*(F)/V^*(\calO) $. Under suitable $ \GLn(\calO)\times\GLm(\calO) $-action, any element in the quotient can be conjugate to 
	\begin{equation}\label{GLn}
		\left(\begin{pmatrix}
			\diag(t^{-a_1},\ldots,t^{-a_r})&0\\0&0
		\end{pmatrix},\begin{pmatrix}
		0&0\\0&\diag(t^{-b_1},\ldots,t^{-b_s})
	\end{pmatrix}\right)
	\end{equation}
	for $ r+s\le n,a_1\ge\cdots\ge a_r\ge 1,b_s\ge\cdots\ge b_1\ge 1 $. 
\end{prop}
\begin{proof}
	By row and column operators on an elements in $ V(F) $, one can make it diagonal, i.e, of the form 
	\begin{equation*}
		\begin{pmatrix}
			\diag(t^{-a_1},\ldots,t^{-a_n})\\0
		\end{pmatrix}
	\end{equation*}
	for $ a_1\ge\cdots\ge a_n $. Let $ r=\max\{i:a_r>0\} $. 
	
	Write $ v^*=(x_{ij})_{1\le i\le n,1\le j\le m} $. The condition $ v^*v\in\glm(\calO) $ and $ vv^*\in\gln(\calO) $ is equivalent to $ x_{ij}\in t^{\max\{a_i,a_j\}}\calO $. Hence $ v^* $ is of the form
	\begin{equation*}
		\begin{pmatrix}
			A_{r,r}&A_{r,m-r}\\A_{n-r,r}&A_{n-r,m-r}
		\end{pmatrix}
	\end{equation*}
	where $ A_{i,j}\in\Mat_{i,j}(F) $ and $ A_{r,r},A_{r,m-r},A_{n-r,r} $ has coefficients in $ t\calO $. 
	
	Next, use $ \GL_{n-r}(\calO)\times\GL_{m-r}(\calO) $ to do row and column operators to make $ A_{n-r,m-r} $ diagonal. Thus we get $ v^*+V^*(\calO) $ is conjugate to $ \begin{pmatrix}
		0&0\\0&\diag(t^{-b_1},\ldots,t^{-b_s})
	\end{pmatrix}+V^*(\calO) $. 
	
	Since $ v\in \begin{pmatrix}
		\diag(t^{-a_1},\ldots,t^{-a_r})&0\\0&0
	\end{pmatrix}+V(\calO) $ and matrices $ \begin{pmatrix}
	1&\\&\GL_{n-r}
\end{pmatrix} $ and $ \begin{pmatrix}
1&\\&\GL_{m-r}
\end{pmatrix} $ fix this set, we know $ v+V(\calO) $ is conjugate to $ \begin{pmatrix}
\diag(t^{-a_1},\ldots,t^{-a_r})&0\\0&0
\end{pmatrix}+V(\calO) $. 
\end{proof}
\begin{cor}
	Let $ n\le m $. The irreducible elements in $ D_{\GLn_\calO\times\GLm_\calO}((T^*V)_F/(T^*V)_\calO) $ is indexed by $ X_\bullet^+(\GLn) $. 
\end{cor}
\begin{proof}
	Just note that the element in (\ref{GLn}) corresponds to $ (a_1,\ldots,a_r,0\ldots,0,-b_1,\ldots,-b_s) $ in $ X_\bullet^+(\GLn)$. 
	
	For $ \lambda\in X_\bullet^+(\GLn) $, consider $ \dV\cG V_\lambda $. 
\end{proof}
\begin{prop}
	Let $ V=\Hom(\bbC^{2n},\bbC^{2m}) $ and $ n\le m $. $ \bbC^{2n}=\bbC^n\oplus(\bbC^n)^* $ is equipped with standard symmetric inner product and $ \bbC^{2m}=\bbC^m\oplus(\bbC^m)^* $ is equipped with standard anti-symmetric inner product. Consider the subset
	\begin{equation*}
		\{v\in V(F):v^*v\in\so_{2n}(\calO),vv^*\in\sp_{2m}(\calO)\}
	\end{equation*}
	and its image in $ V(F)/V(\calO) $. Under suitable $ \Orth_{2n}(\calO)\times\Sp_{2m} $-action, any element in the quotient can be conjugate to 
	\begin{equation*}
		\begin{pmatrix}
			\diag(t^{-a_1},\ldots,t^{-a_r})&0&0\\0&0&0\\0&0&\diag(t^{-b_1},\ldots,t^{-b_s})
		\end{pmatrix}
	\end{equation*}
	for $ r+s\le n, a_1\ge \cdots\ge a_r\ge 1 $, $ b_s\ge\cdots\ge b_1\ge 1 $. 
\end{prop}
\begin{proof}
	Write 
	\begin{equation*}
		v=(v_1,v_2,v_3,v_4)\in\Hom(F^n,F^m)\oplus\Hom(F^n,(F^m)^*)\oplus\Hom((F^n)^*,F^m)\oplus\Hom((F^n)^*,(F^m)^*), 
	\end{equation*}
	and 
	\begin{equation*}
		v^*=(-v_4^\transpose,-v_2^\transpose,v_3^\transpose,v_1^\transpose)\in \Hom(F^m,F^n)\oplus\Hom(F^m,(F^n)^*)\oplus\Hom((F^m)^*,F^n)\oplus\Hom((F^m)^*,(F^n)^*). 
	\end{equation*}
	Then the condition of $ vv^*\in\sp_{2m}(\calO) $ is equivalent to $ v_1v_4^\transpose+v_3v_2^\transpose,v_1v_3^\transpose+v_3v_1^\transpose,v_2v_4^\transpose+v_4v_2^\transpose\in\glm(\calO) $. The condition of $ v^*v\in\so_{2n}(\calO) $ is equivalent to $ v_3^\transpose v_2-v_4^\transpose v_1,v_3^\transpose v_4-v_4^\transpose v_3,v_1^\transpose v_2-v_2^\transpose v_1\in\gln(\calO) $. 
	
	Use elements in $ \GLn(\calO),\GLm(\calO) $ and permutations $ (\bbZ/2\bbZ)^n\ltimes\sym_n,(\bbZ/2\bbZ)^m\ltimes\sym_m $, we can make $ v_1 $ diagonal and $ v_t((v_1)_{jj})\le v_t((v_2)_{ij}) $, $ v_t((v_1)_{ii})\le v_t((v_3)_{ij}) $. 
	
	In particular, write $ v_1=\begin{pmatrix}
		\diag(t^{-a_1},\ldots,t^{-a_n})\\0
	\end{pmatrix} $for $ a_1\ge\cdots\ge a_n $. Let $ r=\max\{i:a_r>0\} $. Write $ v_2 $ as follows
	\begin{equation*}
		\begin{pmatrix}
			t^{-a_1}x_{11}&\cdots&t^{-a_n}x_{1n}\\
			\vdots&&\vdots\\
			t^{-a_1}x_{m1}&\cdots&t^{-a_n}x_{mn}
		\end{pmatrix},
	\end{equation*}
	where $ x_{ij}\in\calO $. Then the condition $ v_1^\transpose v_2-v_2^\transpose v_1\in\gln(\calO) $ gives $ x_{ij}-x_{ji}\in t^{a_i+a_j}\calO,1\le i,j\le n $. 
	
	Take $ y_{ij}=y_{ji}=x_{ji} $ for $ i\le r, i\le j $ and $ y_{ij}=0 $ for $ i,j>r $. This gives an element $ Y $ in $ \Sym^2\calO^m \subset \Sp_{2m}(\calO) $. Take the action, we get $ x_{ij}'=x_{ij}-x_{ji} $ and $ x_{ji}'=0 $ for $ i\le r,i\le j $. Thus $ (v_2')_{ij}=t^{-a_j}(x_{ij}-x_{ji})\in t^{a_i}\calO\subset \calO $ and $ (v_2')_{ji}=0 $ for $ i\le r, i\le j $. When $ i,j>r $, we have $ (v_2')_{ij}=(v_2)_{ij}=t^{-a_j}x_{ij}\in t^{-a_j}\calO\subset\calO $. In conclusion, we have $ v_2'\in\Hom(\calO^n,(\calO^m)^*) $. 
	
	
	Similarly, write 
	\begin{equation*}
		v_3=\begin{pmatrix}
			t^{-a_1}x_{11}&\cdots&t^{-a_1}x_{1m}\\
			\vdots&&\vdots\\
			t^{-a_n}x_{n1}&\cdots&t^{-a_n}x_{nn}\\
			x_{n+1,1}&\cdots&x_{n+1,n}\\
			\vdots&&\vdots\\
			x_{m1}&\cdots&x_{mn}
		\end{pmatrix},
	\end{equation*}
	where $ x_{ij}\in\calO $ for $ 1\le i,j\le n $. The condition $ v_1v_3^\transpose+v_3v_1^\transpose\in\glm(\calO) $ gives $ x_{ij}+x_{ji}\in t^{a_i+a_j}\calO $ for $ 1\le i,j\le n $ and $ x_{ij}\in t^{a_j}\calO $ for $ i>n $. 
	
	If $ a_n\le 0 $, from our construction of $ v_1 $, we know $ x_{ij}\in\calO $ for $ i>n $. Otherwise, we have $ a_1\ge\cdots\ge a_n\ge 1 $, then $ x_{ij}\in t^{a_j}\calO\subset\calO $ for $ i>n $. Anyway, we have $ x_{ij}\in\calO $ for $ i>n $. 
	
	For the remaining, use exactly the same method as before to use an element in $ \Lambda^2\calO^n\subset \SO_{2n}(\calO) $ to make $ v_3\in\Hom((\calO^n)^*,\calO^m) $. 
	
	Now $ v_3 v_2^\transpose\in\glm(\calO),v_3^\transpose v_2\in\gln(\calO) $, we get $ v_1v_4^\transpose\in\glm(\calO), v_4^\transpose v_1\in\gln(\calO) $. Using the result in Proposition \ref{GL index}, we can make $ v_4 $ into a diagonal matrix. 
\end{proof}
\begin{cor}
	Let $ n\le m $. The irreducible elements in $ D_{\Orth_{2n}\times\Sp_{2m}}(V_F/(\VO,\psi)) $ is indexed by $ X_\bullet(\Orth_{2n}) $. 
\end{cor}
\begin{proof}
	As $ r+s\le n $, we can further use permutations in Weyl group to make $ v+V(\calO) $ is conjugate to $ \begin{pmatrix}
		\diag(t^{-a_1},\ldots,t^{-a_r})&0\\0&0
	\end{pmatrix}+V(\calO) $ for $ r\le n $. Thus it corresponds to $ (a_1,\ldots, a_r,0\ldots,0)\in X_\bullet(\Orth_{2n}) $. 
\end{proof}
\section{Deequivariantization}
\subsection{Hecke action on the lattice model}
For a group homomorphism $ G\rightarrow \Mp(V) $, we want to define the action of $ D(G) $ on $ D(V/(L,\psi)) $, we need a kernel sheaf on $ G\times V $. This is done in \cite{SR07'} and also \cite{LL09}. Let $ \LagV $ be the space of all Lagrangians on $ V $, \cite{LL09} constructed a sheaf $ \calF_\LagV $ on $ \LagV\times\LagV\times V $ with properties. By the map $ G\rightarrow \LagV\times\LagV $ given by $ g\mapsto (gL,L) $, we obtain a sheaf $ \FG $ on $ G\times V $. Thus we can define the action by
\begin{equation*}
	\calS*\calF=\act_!(\pr_2^*\FG\otimes \pr_{23}^*\calS\otimes\pr_{13}^*\calF\otimes\calL_\psi), 
\end{equation*}
Here $ \act\colon G\times V\times V\rightarrow V $ is given by $ (g,v_1,v_2)\mapsto gv_1+v_2 $; $ \pr $ are corresponding projections; $ \calL_\psi $ is the sheaf on $ G\times V\times V$ given by the pullback of Artin-Schreier sheaf through $G\times V\times V\rightarrow \bbA^1, (g,v_1,v_2)\mapsto \omega(gv_1,v_2) $. 

The properties of $ \calF_\LagV $ ensures this action gives a genuine module structure. 

For the unit, take $ \calS=\delta_1\in D(G) $. From the property $ \Delta^*\calF_\LagV=\calF_\Delta $, we know $ \FG|_1=\Lambda_L $ and thus the convolution product with an $ (L,\psi) $-equivariant sheaf is just identity. 

\begin{prop}
	The associativity holds. I.e., we have $ \calS_1*(\calS_2*\calF)\simeq (\calS_1*\calS_2)*\calF $. 
\end{prop}
\begin{proof}
	For clarity, we use $ (g_1,g_2v_1+v_2,v_3) $ to denote the map $ G\times G\times V\times V\times V\rightarrow G\times V\times V $ given by $ (g_1,g_2,v_1,v_2,v_3)\mapsto (g_1,g_2v_1+v_2,v_3) $ and similarly for other maps. Then we have
	\begin{align*}
		\calS_1*(\calS_2*\calF)=&(g_1(g_2v_1+v_2)+v_3)_!((g_1,g_2,v_1)^*(\calS_1\boxtimes\calS_2\boxtimes\calF)\otimes\\
		&\otimes(g_2,v_2)^*\FG\otimes (g_1,v_3)^*\FG\otimes \omega(g_2v_1,v_2)^*\calL_\psi\otimes\omega(g_1(g_2v_1+v_2),v_3)^*\calL_\psi).
	\end{align*}
	
	From the convolution property of $ \calF_\LagV $, we have the following isomorphism in $ \LagV\times\LagV\times\LagV\times V $:
	\begin{equation*}
		\add_!(\pr_{15}^*\calF_\LagV\otimes\pr_{34}^*\calF_\LagV\otimes\calL_\psi)\simeq \pr_2^*\calF_\LagV.
	\end{equation*}
	Take the pullback by the map $ G\times G\rightarrow \LagV\times\LagV\times\LagV, (g_1,g_2)\mapsto(g_1g_2L,g_1L,L) $, we see 
	\begin{equation*}
		\add_!((g_1,v_1)^*\FG\otimes(g_2,g_1^{-1}v_2)^*\FG\otimes\calL_\psi)\simeq \mult^*\FG.
	\end{equation*}
	Here, we used the fact that $ \calF_\LagV $ is $ G $-equivariant. By change of variables, we see
	\begin{equation*}
		(g_1,g_2,v_1+g_1v_2)_!((g_1,v_1)^*\FG\otimes(g_2,v_2)^*\FG\otimes\omega(v_1,g_1v_2)^*\calL_\psi)\simeq\mult^*\FG. 
	\end{equation*}
	Hence we can simplify, by letting $ u=g_1v_2+v_3 $, 
	\begin{align*}
		\calS_1*(\calS_2*\calF)=&(g_1g_2v_1+g_1v_2+v_3)_!((g_1,g_2,v_1)^*(\calS_1\boxtimes\calS_2\boxtimes\calF)\otimes\\
		&\otimes(g_2,v_2)^*\FG\otimes(g_1,v_3)^*\FG\otimes\omega(g_1g_2v_1,g_1v_2+v_3)^*\calL_\psi\otimes\omega(g_1v_2,v_3)^*\calL_\psi)\\
		=&(g_1g_2v_1+u)_!((g_1,g_2,v_1)^*(\calS_1\boxtimes\calS_2\boxtimes\calF)\otimes(g_1g_2,u)^*\FG\otimes\omega(g_1g_2v_1,u)^*\calL_\psi)\\
		=&(gv_1+u)_!((g,v_1)^*((\calS_1*\calS_2)\boxtimes\calF)\otimes(g,u)^*\FG\otimes\omega(gv_1,u)^*\calL_\psi).
	\end{align*}
	The right hand side is exactly $ (\calS_1*\calS_2)*\calF $. 
\end{proof}

The image of an $ (L,\psi) $-equivariant sheaf is still an $ (L,\psi) $-equivariant sheaf comes from the $ \act_{lr} $-equivariant property of $ \calF_\LagV $. 

If a subgroup $ H\subset G $ fixes $ (L,\psi) $, we get the map $ G/H\rightarrow \LagV\times\LagV $, using it, we can define the action of $ D(H\backslash G/H) $ on $ D_H(V/(L,\psi)) $ similarly: 
\begin{equation*}
	\calS*\calF=\act_!(\pr_2^*\FG\otimes \pr_3^*(\calS\widetilde{\boxtimes}\calF)\otimes\calL_\psi),
\end{equation*}
Here $ \act\colon H\backslash ((G\times^H V)\times V)\rightarrow H\backslash V $ is given by $ (g,v_1,v_2)\mapsto gv_1+v_2 $. Since $ H $ fixes $ L $, $ \FG $ descends to a sheaf $ \calF_{G/H} $ on $ G/H\times V $. The $ \act_G $-equivariant property of $ \calF_\LagV $ ensures $ \calF_{G/H} $ is $ H $-equivariant under the action of $ h\cdot(gH,v)=(hgH,hv) $. In conclusion, the action $ \calS*\calF $ is well-defined. The proof of properties such as associativity is identical as above. 

\begin{prop}
	Take a subspace $ W\subset L $ and subgroup $ H\subset G $ that fixes $ W $. Then $ H $ acts on the symplectic space $ W^\perp/W $. We have the compatibility of both actions: 
	\begin{equation*}
		\begin{tikzcd}
			D(H)\dar&\otimes& D((W^\perp/W)/(L/W,\psi))\dar\rar{\act}&D((W^\perp/W)/(L/W,\psi))\dar\\
			D(G)&\otimes&D(V/(L,\psi))\rar{\act}&D(V/(L,\psi))
		\end{tikzcd}
	\end{equation*}
\end{prop}
The compatibility of $ \calF_\LagV $ under taking a subquotient $ W^\perp/W $ of a Lagrangian $ W\subset V $ ensures the actions 
\begin{equation*}
	\Sat_{G}^{(n)}\otimes D_{\GO}((t^{-r}\VO/t^r\VO)/(\VO/t^r\VO,\psi))\rightarrow D_{\GO}((t^{-r-n}\VO/t^{r+n}\VO)/(\VO/t^{r+n}\VO,\psi))
\end{equation*}
are compatible. In conclusion, we have the action of $ \Sat_G $ on $ D_{\GO}(V_F/(\VO,\psi)) $. 

For our cases, $ G\times H\rightarrow\Sp(V)$ has a lift to $ \Mp(V) $, we obtain a $ D_{\GO\times \HO}(\Gr_{G\times H}) $-action on $ D_{\GO\times\HO}(V_F/(\VO,\psi)) $. 

\subsection{Through deequivariantized Hecke category}
Let $\OS=\Hom(\dV,\dV\cG\calO(\Gl))$. From \cite{BF08}, we have the isomorphism $\calO(\gl^*)=\Sym(\gl[-2])=\Hom(\dG,\dG\cG\calO(\Gl))$ and similarly $\calO(\hl^*)=\Sym(\hl[-2])=\Hom(\dH,\dH\cH\calO(\Hl))$. We define the following maps: 
\begin{equation*}
	\Sym(\gl[-2])\rightarrow \OS \quad\mbox{and}\quad \Sym(\sh{\hl})\rightarrow \OS. 
\end{equation*}
%

From the shearing on $\hl$, we have the map $\gl[-2] \rightarrow \sh{\hl}$. We will show the following diagram commute: 
\begin{equation*}
	\begin{tikzcd}
		\Sym(\sh{\hl})\rar&\Hom(\dV,\dV\cG\Res(\calO(\Hl)))\dar\\
		\Sym(\gl[-2])\uar\rar&\OS=\Hom(\dV,\dV\cG\calO(\Gl))
	\end{tikzcd}
\end{equation*}

The first map $\Hom(\dG,\dG\cG\calO(\Gl))\rightarrow \Hom(\dV,\dV\cG\calO(\Gl))$ is just defined via the action of Satake category on the category of Weil representation. 

\begin{lem}
	$ \Sym(\sh{\hl})\simeq\bigoplus\limits_{W\in\Irr\Hl}\Hom(\dG,\IC_{W^*})\otimes\gr{W}. $
\end{lem}
\begin{proof}
	From \cite{BF08}, we have $ \Sym^i\hl\simeq\bigoplus\limits_{W\in\Irr\Hl}\Ext^{2i}(\dH,\IC_{W^*})\otimes W $ as $ \Hl $ representations. Thus we can apply the grading of elements in the Cartan subgroup: 
	\begin{equation*}
		\Sym^i\gr{\hl}\simeq\bigoplus\limits_{W\in\Irr\Hl}\Ext^{2i}(\dH,\IC_{W^*})\otimes \gr{W}, 
	\end{equation*}
	hence 
	\begin{equation*}
		\Sym(\sh{\hl})=\Sym(\gr{\hl}[-2])\simeq\bigoplus\limits_{W\in\Irr\Hl}\Hom(\dH,\IC_{W^*})\otimes \gr{W}. \qedhere
	\end{equation*}
\end{proof}

The generators of this algebra is $ \sh{\hl}=\bigoplus\limits_{\Irr\Hl\ni W\subset\hl}\Ext^2(\dH,\IC_{W^*})\otimes\gr{W}[-2] $. 

We have maps from the Hecke action:
\begin{equation*}
	\Ext^{2i}(\dH,\IC_{W'^*})\otimes\gr{W'}\rightarrow \Ext^{2i}(\dV,\dV\cH W'^*)\otimes\gr{W'}\simeq\Ext^{2i}(\dV,\dV\cG\gRes(W'^*))\otimes\gr{W'}. 
\end{equation*}
Let $\gRes(W'^*)=\bigoplus\limits_{W\in\Irr\Gl}W\otimes M_W$, where $M_W$ is a graded vector space associated to the multiplicity of $W$. Then $\gr{W'}=\bigoplus\limits_{W\in\Irr\Gl}W^*\otimes M_W^*$. 

Hence we get the direct summand 
\begin{equation*}
	\bigoplus\limits_{W\in\Irr\Gl}\Ext^{2i}(\dV,\dV\cG W\otimes M_W)\otimes W^*\otimes M_W^*\subset \Ext^{2i}(\dV,\dV\cG\gRes(W'^*))\otimes\gr{W'}. 
\end{equation*}

Write $M_W=\oplus_{k\in\bbZ}M_{W,k}[k]$ and $M_W^*=\oplus_{k\in\bbZ}M_{W,k}^*[-k]$. Thus the first term has the direct summand 
\begin{equation*}
	\bigoplus_{k\in\bbZ}\Ext^{2i}(\dV,\dV\cG W\otimes M_{W,k}[k])\otimes W^*\otimes M_{W,k}[-k]=\bigoplus_{k\in\bbZ}\Ext^{2i+k}(\dV,\dV\cG W\otimes M_{W,k})\otimes W^*\otimes M_{W,k}[-k]. 
\end{equation*}
Taking traces of each $M_W$, we obtain a map to $ \Ext^{2i+k}(\dV,\dV\cG W)\otimes W^*[-k] $. 

In conclusion, we get the map 
\begin{equation*}
	\Ext^{2i}(\dH,\IC_{W'^*})\otimes\gr{W'}\rightarrow \bigoplus_{k\in\bbZ}\bigoplus_{W\in\Irr\Gl}\Ext^{2i+k}(\dV,\dV\cG W)\otimes W^*[-k]=\Ext^{2i}(\dV,\dV\cG\calO(\Gl)), 
\end{equation*}
and hence 
\begin{equation*}
	\Sym(\sh{\hl})=\bigoplus_{W'\in\Irr\Hl}\bigoplus_{i\in\bbZ}\Ext^{2i}(\dH,\IC_{W'^*})\otimes\gr{W'}[-2i]\rightarrow\bigoplus_{i\in\bbZ}\Ext^{2i}(\dV,\dV\cG\calO(\Gl))[-2i]=\OS. 
\end{equation*}

\subsection{Algebraic map} 
Under the assumption that $ \gRes $ is monoidal, we have this map is an algebraic homomorphism. 

On the other hand, we have a map $ \hl^\shear\hookrightarrow \Sym(\hl^\shear) \rightarrow \OS $. By showing $ \OS $ is commutative, we obtain another map $ \Sym(\hl^\shear) \rightarrow \OS $. But without the assumption, it is not known if this map coincide with the map defined before. 
\section{Examples}
\subsection{The action by standard representations}\label{std}
Here I give an explicit calculation of $\dV\cn\std_n$ in the $\GLn\times\GLm$ case. 

Note that $\Gr_{\GLn,e_1}
=\Pn$, we have $\IC_\std=\bbC_\Pn$. Thus $\dV\cn\std_n$ is the pushforward of the constant sheaf on $\VO\wtimes\Gr_{\GLn,e_1}$ to $\VF$. 

Since $\Gr_{\GLn,e_1}$ parameterize lattices $\Lambda$ such that $\calO^n\subset\Lambda\subset (t^{-1}\calO)^n$ and $\dim\Lambda/\calO^n=1$, by definition $\VO\wtimes\Gr_{\GLn,e_1}$ parameterize such a lattice $\Lambda$ and $m$ vectors in this lattice, and the map $\VO\wtimes\Gr_{\GLn,e_1}\rightarrow \VF$ forgets this lattice. 

Hence the image lies in $t^{-1}\VO$. For any element in $\VO$, its preimage is the whole $\Pn$. For any element in the image and not in $\VO$, the preimage is just one point. Thus $\dV\cn\std_n$ can be viewed as a sheaf on $t^{-1}\VO/\VO=V$. In this viewpoint, the support is the elements in $V$ whose rank is less or equal to $1$. the stalk at rank $1$ is $\bbC$, and the stalk at $0$ is $H^*(\Pn)[m+n-1]$. 

Besides, we can calculate the intersection complex directly. A rank $1$ matrix can be written as the product of a non-zero row vector and a non-zero column vector. Thus the open part is $(\bbC^n\backslash\{0\}\times\bbC^m\backslash\{0\})/\Gm$, or the $\bbC^*$ bundle $\calO(-1,-1)$ on $\Pn\times\Pm$. Then the whole space is the affine cone of this line bundle. Now the stalk of the intersection complex at $0$ is 
\begin{equation*}
	(\IC_{e_1})_0=\tau_{\le-1}(H^*(\Pn\times\Pm)/c_1(\calO(-1,-1))[m+n-1]).
\end{equation*}
Here, $\tau$ is the truncation functor related to the classical t-structure, and quotient means the taking the cone of the map 
\begin{equation*}
	c_1(\calO(-1,-1)): H^*(\Pn\times\Pm)\rightarrow H^*(\Pn\times\Pm)[-2]. 
\end{equation*}
When $n<m$, this turns out to be isomorphic to $H^*(\Pn)[m+n-1]$ and also $(\tau_{\le 2n-2}H^*(\Pm))[m+n-1]$. Thus we see $\dV\cn\std_n\simeq \IC_{e_1}$. 

Similarly, $\dV\cm\std_m$ supports on the same set and the stalk at $0$ is $H^*(\Pm)[m+n-1]$. The decomposition theorem says that this complex is a direct sum of simple objects. Besides $\IC_{e_1}$, the remaining support at $0$ and the stalk is $(\tau_{\ge 2n}H^*(\Pm))[m+n-1]$. Hence is $\bbC_0[m-n-1]\oplus\bbC_0[m-n-3]\cdots\oplus\bbC_0[-m+n+1]$. 

This calculation verifies the result in \cite{Ly11} that $\dV\cm\std_m\simeq\dV\cn\gRes(\std_m)$. 

From this viewpoint, it is clear that the action of $H^2(\Pn)=\Ext^2(\IC_{\std_n})$ on $(\dV*\IC_{\std_n})_0$ coincide with the action of $H^2(\Pm)=\Ext^2(\IC_{\std_m})$ on $(\dV*\IC_{\std_n})_0$. 

In fact, we can calculate $\End(\dV\cn\std_n)$ directly. $\dV\cn\std_n=\tau_{\le-1}j_*\bbC[m+n-1]$ fits into an exact triangle: 
\begin{equation*}
	\tau_{\le-1}j_*\bbC[m+n-1]\rightarrow j_*\bbC[m+n-1]\rightarrow \tau_{\ge 0}j_*\bbC[m+n-1].
\end{equation*}
Hence an exact triangle 
	\begin{multline*}
		\Hom(\tau_{\le-1}j_*\bbC[m+n-1],\tau_{\le-1}j_*\bbC[m+n-1])\rightarrow\\
		\rightarrow \Hom(\tau_{\le-1}j_*\bbC[m+n-1],j_*\bbC[m+n-1])\rightarrow \Hom(\tau_{\le-1}j_*\bbC[m+n-1],\tau_{\ge 0}j_*\bbC[m+n-1]).
	\end{multline*}
We can calculate
\begin{align*}
	&\Hom(\tau_{\le-1}j_*\bbC[m+n-1],\tau_{\le-1}j_*\bbC[m+n-1])\\
	=&\Hom(\tau_{\le-1}j_*\bbC[m+n-1],j_*\bbC[m+n-1])\\
	=&\Hom(j^*\tau_{\le-1}j_*\bbC[m+n-1],\bbC[m+n-1])\\
	=&\Hom(\bbC[m+n-1],\bbC[m+n-1])\\
	=&H^*(\Pn\times\Pm)/c_1(\calO(-1,-1)),
\end{align*}
and 
\begin{align*}
	&\Hom(\tau_{\le-1}j_*\bbC[m+n-1],\tau_{\ge 0}j_*\bbC[m+n-1])\\
	=&\Hom(\tau_{\le-1}j_*\bbC[m+n-1],i_*H^*(\Pn)[n-m])\\
	=&\Hom(H^*(\Pn)[m+n-1],H^*(\Pn)[n-m]). 
\end{align*}

When $n=1$, $\dV\cn\std_n=\bbC[m]$, it is clear the endomorphism is $\bbC$ of degree $0$. When $m>n\ge 2$, the minimum degree of the complex $\Hom(H^*(\Pn)[m+n-1],H^*(\Pn)[n-m])$ is $2m-1-(2n-2)\ge3$. Thus we have 
\begin{equation*}
	\Ext^2(\tau_{\le-1}j_*\bbC[m+n-1],\tau_{\le-1}j_*\bbC[m+n-1])\simeq\Ext^2(\tau_{\le-1}j_*\bbC[m+n-1],j_*\bbC[m+n-1]). 
\end{equation*}

By tracking the action, the map from $H^*(\Pn)=\End(\IC_{\std_n})$ and $H^*(\Pm)=\End(\IC_{\std_m})$ to $H^*(\Pn\times\Pm)/c_1(\calO(-1,-1))$ is the canonical map. Thus the images of $H^2(\Pn)$ and $H^2(\Pm)$ are the same. Furthermore, the image to $\Ext^2(\dV\cn\std_n,\dV\cn\std_n)$ is the same. 
\subsection{The case when $ n=1 $}
In this case, the vector space $ V $ is the standard representation of $ \GLm $. 

The closed orbits of $ \GLm_\calO $ on $ \VF $ are given by integers $ k\in \bbZ $. In particular, if $ v=t^{k}v_{k}+t^{k+1}v_{k+1}+\cdots $ is in an orbit, then $ t^{k}\VO $ is in this closed orbit. 

It is clear that $ t^{k}\VO $ is smooth, thus we have $ \IC_k=\bbC_{t^{k}\VO}[-km] $. 

Let $ A=H^*_{\Gm\times\GLm}(\pt)=\bbC[a,e_1,\ldots,e_n] $. When $ k\le l $, we have $ \Hom(\IC_k,\IC_l)=\Hom(\IC_0,\IC_{l-k})=A[-(l-k)m] $. When $ k\ge l $, we have $ \Hom(\IC_k,\IC_l)=\Hom(\IC_0,\IC_{l-k})=i_0^!\IC_{\bbA^{(k-l)m}}=A[-(k-l)m] $. 

Thus $ \OS=\Hom(\dV,\dV\conv_{\Gm}\calO(\Gm))=\Hom(\IC_0,\oplus_k\IC_k)=\oplus_k A[-|k|m] $. 

Next, we determine the algebraic structure of this algebra. As $ \OS $ is an $ A $-algebra, we have $ \OS=A\oplus\bigoplus_{k\ge 1}Ax_k\oplus\bigoplus_{k\ge 1}Ay_k $, where $ \deg x_k=\deg y_k=km $. 

From the algebraic structure of $ \OS $, we know $ x_k\cdot x_l\in Ax_{k+l} $. By degree reason, we see $ x_k\cdot x_l=c_{k,l} x_{k+l} $ for some number $ c_{k,l} $. It is clear the product is not zero, we see $ x_k $ are generated by $ x=x_1 $. Similarly we have $ y_k $ are generated by $ y=y_1 $. 

Finally, to determine the product of $ x $ and $ y $ we need the following result: 
\begin{lem}
	Let $\tau\colon E\rightarrow Z $ be a complex vector bundle of dimension $ n $, $ i\colon Z\rightarrow E $ the zero-section embedding. 
	\begin{enumerate}
		\item We have the isomorphism $ H^*(Z,\bbC_Z)\simeq \Hom(\bbC_Z,\bbC_E[2n]) $. The image of $ 1\in H^0(Z,\bbC_Z) $ is the Thom class. 
		\item The adjoint map $ \bbC_E \rightarrow i_*i^*\bbC_E=i_*\bbC_Z $ induces the map $ \Hom(\bbC_Z,\bbC_E)\rightarrow H^*(E,\bbC_E)\simeq H^*(Z,\bbC_Z) $. The image of Thom class in $ H^{2n}(Z,\bbC_Z) $ is the Euler class. 
	\end{enumerate}
\end{lem}
\begin{proof}
	See Exercise III.7 in \cite{KS90}. Note that complex vector bundles ensure the required orientation. 
\end{proof}
Apply this lemma to the vector bundle $ \std/(\Gm\times\GLm)\rightarrow \pt/(\Gm\times\GLm) $, we see the product $ x\cdot y\in A=H^{2m}_{\Gm\times\GLm}(\pt) $ is the Euler class of $ \std $, which is $ a^n+e_1a^{n-1}+\cdots+e_n $. 

In conclusion, we see $ \OS=\bbC[a,e_1,\ldots,e_{n-1},x,y] $. 
\section{Localization}
In this section, we prove the main theorem about equivalent categories. 
\subsection{Pass through the Slodowy slice}
Consider the case $ (G,H)=(\GLn,\GLm) $. Then the map $ \hl^\shear \rightarrow \OS $ can be explicitly written down. 

After suitable shift of degrees, it is the map 
\begin{align*}
	&\glm\otimes\Ext^2(\dH,\dH\cH\glm)\rightarrow\glm\otimes\Ext^2(\dV,\dV\cm\glm)\\	=&\glm\otimes\Ext^2(\dV\cm\std_m,\dV\cm\std_m)\simeq\glm\otimes\Ext^2(\dV\cn\gRes\std_m,\dV\cn\gRes\std_m)\\
	\simeq&(\gln\oplus \std_n^{\oplus(m-n)}\oplus\std_n^{\vee\oplus (m-n)}\oplus\bbC^{(m-n)^2})\otimes\Ext^2(\dV\cn\std_n\oplus\dV^{\oplus(m-n)\shear},\dV\cn\std_n\oplus\dV^{\oplus(m-n)\shear})\\
	\rightarrow &\gln\otimes\Ext^2(\dV\cn\std_n,\dV\cn\std_n)\oplus\std_n^{\oplus(m-n)}\otimes\Ext^2(\dV\cn\std_n,\dV^{\oplus(m-n)\shear})\oplus\\
	&\oplus\std_n^{\vee\oplus (m-n)}\otimes\Ext^2(\dV^{\oplus(m-n)\shear},\dV\cn\std_n)\oplus\bbC^{(m-n)^2}\otimes\Ext^2(\dV^{\oplus(m-n)\shear},\dV^{\oplus(m-n)\shear})\\
	\rightarrow &\gln\otimes\Ext^2(\dV\cn\std_n,\dV\cn\std_n)\oplus \std_n\otimes\bigoplus_i\Ext^{2+i}(\dV\cn\std_n,\dV)\oplus\\
	&\oplus\std_n^{\vee}\otimes\bigoplus_i\Ext^{2+i}(\dV,\dV\cn\std_n)\oplus\bigoplus_i\Ext^{2+i}(\dV,\dV). 
\end{align*}
The index $ i $ in first and second direct sum runs through $ m-n-1, m-n-3, \ldots, 1-m+n $. The index $ i $ in the last direct sum runs through $ 2(m-n-1), 2(m-n-1)-2, \ldots, 2(1-m+n) $. 

In order to show the image of $ \glm $ pass through only one $ \std_n $ and one $ \std_n^\vee $, we need the following result:
\begin{prop}
	We have the following vanishing result:
	\begin{align*}
		&\Ext^{<m-n+1}(\dV\cn\std_n,\dV)=\Ext^{<m-n+1}(\dV,\dV\cn\std_n)=0,\\
		&\dim\Ext^{m-n+1}(\dV\cn\std_n,\dV)=\dim\Ext^{m-n+1}(\dV,\dV\cn\std_n)=1.
	\end{align*}
\end{prop}
\begin{proof}
	Note that both sheaves come from $ V_F/\VO $. Indeed, they support inside $ t^{-1}\VO/\VO=V $. Let $ i\colon\pt \rightarrow V $ be the embedding of zero point. Then
	\begin{equation*}
		\Hom(\dV\cn\std_n,\dV)=\Hom(i^*(\dV\cn\std_n),\delta_\pt). 
	\end{equation*}
	Note that $ \dV\cn\std_n $ is an intersection complex, and thus is self dual. Hence
	\begin{equation*}
		\Hom(\dV,\dV\cn\std_n)=i^!(\dV\cn\std_n)=D(i^*(\dV\cn\std_n)).
	\end{equation*}
	From the calculation in the previous section, we know $ i^*(\dV\cn\std_n) $ is $ H^*(\Pn)[m+n-1] $. With equivariant structure, it is $ H^*_\GLm(\pt)\otimes H^*_\GLn(\Pn)[m+n-1] $. 
	
	Taking the dual in $ D(\pt/(\GLm\times\GLn))=H^*_\GLm(\pt)\otimes H^*_\GLn(\pt)-\mathrm{mod} $, and using the fact that $ H^*_\GLn(\Pn) $ is a free $ H^*_\GLn(\pt) $-module, we see the resulting $ \Hom $ is 
	\begin{equation*}
		H^*_\GLm(\pt)\otimes H^*_\GLn(\pt)[1-m-n]\oplus H^*_\GLm(\pt)\otimes H^*_\GLn(\pt)[3-m-n]\oplus\cdots\oplus H^*_\GLm(\pt)\otimes H^*_\GLn(\pt)[n-m-1]. 
	\end{equation*}
	This proves the required proposition. 
\end{proof}
As for the part $ \bbC^{(m-n)^2}\subset\glm $, first note that $ \Ext^{<0}(\dV,\dV)=0 $, the only remaining parts is those with weights greater or equal to $ -2 $. 

More precisely, let the image of $ \Ext^2(\dH,\dH\cH\glm) $ to $ \Ext^2(\dV^{\oplus(m-n)\shear},\dV^{\oplus(m-n)\shear}) $ be 
\begin{equation*}
	f_{i,j}\in\Ext^{2(j-i+1)}(\dV,\dV), 1\le i,j\le m-n.
\end{equation*}
Then $ f_{i,j}=0 $ for $ i>j+1 $, and $ f_{i+1,i} $ are some constant numbers. 

In order to show the image of $ \glm $ pass through the Slodowy slice, we need the following result:
\begin{prop}
	For fixed $ j-i=k\ge0 $, all the $ f_{i,j}\in\Ext^{2k+2}(\dV,\dV) $ are colinear to each other. 
\end{prop}
\begin{proof}
	As an action on $ \dV\cn\std_m $, the element $ f\in\Ext^2(\dV\cn\std_m,\dV\cn\std_m) $ always commutes with itself. 
	
	In particular, $ f $ must commute with its degree $ 0 $-part. For the 1-parameter family $ f^{(t)} $, we should have $ ff^{(t)}=f^{(t)}f $. Take the limit $ t\rightarrow 0 $, we see $ ff^{(0)}=f^{(0)}f $. 
	
	Now $ f^{(0)} $ are the isomorphisms $ f_{i+1,i} $. So we have the commutative diagrams:
	\begin{equation*}
		\begin{tikzcd}
			\dV\rar{f_{i,j}}\dar{f_{j+1,j}}&\dV[2k+2]\dar{f_{i+1,i}}\\
			\dV\rar{f_{i+1,j+1}}&\dV[2k+2]
		\end{tikzcd}
	\end{equation*}
	Hence $ f_{i,j} $ are colinear for fixed $ j-i=k $. 
\end{proof}
\subsection{Compatibility of two actions}
The action of $ \GLn $ comes from the map 
\begin{equation*}
	\gln\otimes\Ext^2(\dG,\dG\cG\gln)\rightarrow \gln\otimes\Ext^2(\dV,\dV\cG\gln)=\gln\otimes\Ext^2(\dV\cG\std_n,\dV\cG\std_n).
\end{equation*}
As calculated in Section \ref{std}, $ \Ext^2(\dV\cG\std_n,\dV\cG\std_n) $ has only one dimension after dividing $ \Ext^2(\dV,\dV) $. Hence the image of $ \Ext^2(\dH,\dH\cH\glm) $ to $ \Ext^2(\dV\cG\std_n,\dV\cG\std_n) $ is a multiple of the image of $ G $. Hence we can have the compatibility:
\begin{equation*}
	\begin{tikzcd}
		\Sym(\gl[-2])\dar\rar&\OS\\
		\Sym(\hl[-2])\rar&\calO(S_f)\uar\\
	\end{tikzcd}
\end{equation*}

In particular, the map $ \calO(S_f)\rightarrow \OS $ is a $ H^*_{G\times H}(\pt) $-algebra homomorphism. 
\subsection{Linear algebra}
We calculate the fiber of the map $S_f\rightarrow \gl^*\sslash\Gl\times\hl^*\sslash\Hl$ up to codimension one. 

Elements in $ S_f $ look like
\begin{equation*}
	\begin{pmatrix}
		x&&&&&v\\
		v^*&a_1&a_2&a_3&\cdots&a_{m-n}\\
		&c_1&a_1&a_2&\cdots&a_{m-n-1}\\
		&&\ddots&\ddots&\ddots&\vdots\\
		&&&c_{m-n-2}&a_1&a_2\\
		&&&&c_{m-n-1}&a_1
	\end{pmatrix}, x\in\gln,v\in\std_n,v^*\in\std_n^*,
\end{equation*}
where $ c_i $ are positive constants. Its image is given by characteristic polynomial of $ x $ and this whole matrix. The later is calculated by
\begin{equation*}
	\chi_x(\lambda)(\lambda^{m-n}-d_1a_1\lambda^{m-n-1}+\cdots+(-1)^{m-n} d_{m-n}a_{m-n}+d_{m-n+1}v^*(\lambda I-x)^{-1}v). 
\end{equation*}
Here $ d_i $ are positive constants. 
\begin{prop}
	The fiber at a generic point of $\gl^*\!\sslash\Gl\times\hl^*\!\sslash\Hl$ is isomorphic to $\Gl$. 
\end{prop}
\begin{proof}
	Given two polynomials $ f(\lambda),g(\lambda) $ of degree $ n $ and $ m $. If the discriminant of $ f $ and the resultant of $ f $ and $ g $ are non-zero, we will show the fiber is $ \GLn $. 
	
	Write $ g=qf+r $ such that $ \deg r<n $. Then we know $ q(\lambda)=\lambda^{m-n}-d_1a_1\lambda^{m-n-1}+\cdots+(-1)^{m-n} d_{m-n}a_{m-n} $, which shows that $ a_i $ are fixed. 
	
	Since $ x $ has characteristic polynomial $ \chi_x(\lambda)=f(\lambda) $ with distinct roots, $ x $ is conjugated to a diagonal matrix by an element in $ \GLn $. Write $ x=g\diagmat(\lambda_1,\ldots,\lambda_n)g^{-1} $. Then $ v^*(\lambda I-x)^{-1}v=\sum (v^*g)_i(g^{-1}v)_i\frac{1}{\lambda-\lambda_i} $. Hence we have 
	\begin{equation*}
		(v^*g)_i(g^{-1}v)_i=e_i:=\frac{r(\lambda_i)}{d_{m-n+1}\prod_{j\neq i}(\lambda_i-\lambda_j)}. 
	\end{equation*}

	By taking an action of a diagonal matrix, $ x $ is unchanged, and we can make $ (g^{-1}v)_i=1 $. 
	In conclusion, $ x=g\diagmat(\lambda_1,\ldots,\lambda_n)g^{-1},v=g(1,\ldots,1)^\transpose,v^*=(e_1,\ldots,e_n)g^{-1} $ gives all the possible fibers and they are different. 
\end{proof}
\begin{prop}
	The generic fiber at the hyperplane given by resultant is isomorphic to $ (\GLn\times+)/\Gm $. 
\end{prop}
\begin{proof}
	Given two polynomials $ f(\lambda), g(\lambda) $ of degree $ n $ and $ m $. Assume the discriminant of $ f $ is non-zero and the resultant of $ f $ and $ g $ is zero. 
	
	Write $ g=qf+r $ such that $ \deg r<n $. Then we know $ q(\lambda)=\lambda^{m-n}-d_1a_1\lambda^{m-n-1}+\cdots+(-1)^{m-n} d_{m-n}a_{m-n} $, which shows that $ a_i $ are fixed. 
	
	Write $ x=g\diag(\lambda_1,\ldots,\lambda_n)g^{-1} $ for $ g\in\GLn $. The condition of the resultant implies $ r(\lambda_{i_0})=0 $ for one $ i_0\in\{1,\ldots,n\} $. 
	
	We have the same formula as before. When $ i\neq i_0 $, $ (v^*g)_i $ and $ (g^{-1}v)_i $ has a single orbit under $ \Gm $-action. When $ i=i_0 $, $ (v^*g)_i $ and $ (g^{-1}v)_i $ form the space $ +=\Spec\bbC[X,Y]/(XY) $. 
	
	In conclusion, we have a surjective map from $ \GLn\times+\times\Gm^{n-1} $ to the fiber, which is stable under the action of diagonal matrices in $ \GLn $. Hence we get the isomorphism of the fiber and the space $ (\GLn\times+)/\Gm $. 
\end{proof}
\begin{prop}
	The generic fiber at the hyperplane given by root hyperplanes of $ \gln\sslash\GLn $ is isomorphic to $ \GLn $. 
\end{prop}
\begin{proof}
	Now we have two polynomials $ f(\lambda) $ and $ g(\lambda) $ of degree $ n $ and $ m $ with non-zero resultant, but $ f $ has a double root $ \lambda_0 $. 
	
	Write $ g=qf+r $ such that $ \deg r<n $. Then we know $ q(\lambda)=\lambda^{m-n}-d_1a_1\lambda^{m-n-1}+\cdots+(-1)^{m-n} d_{m-n}a_{m-n} $, which shows that $ a_i $ are fixed. 
	
	If $ x $ is conjugate to a diagonal matrix, then $ v^*(\lambda I-x)^{-1}v=\sum(v^*g)_i(g^{-1}v)_i\frac1{\lambda-\lambda_i} $, implying $ r(\lambda_0)=0 $. This contradicts with the resultant condition. 
	
	Thus $ x $ must conjugate to
	\begin{equation*}
		x_0=\begin{pmatrix}
			\lambda_0&1&&&\\
			&\lambda_0&&&\\
			&&\lambda_3&&\\
			&&&\ddots&\\
			&&&&\lambda_n
		\end{pmatrix}.
	\end{equation*}
	Then 
	\begin{align*}
		v^*(\lambda I-x)^{-1}v=&\sum_{i\ge3}(v^*g)_i(g^{-1}v)_i\frac1{\lambda-\lambda_i}+\\&+\big((v^*g)_1(g^{-1}v)_1+(v^*g)_2(g^{-1}v)_2\big)\frac1{\lambda-\lambda_0}-(v^*g)_1(g^{-1}v)_2\frac1{(\lambda-\lambda_0)^2}.
	\end{align*}
	Hence for $ i\ge 3 $, we still have $ (v^*g)_i(g^{-1}v)_i $ is a fixed non-zero number. Also, 
	\begin{equation*}
		(v^*g)_1(g^{-1}v)_2=e_1:=-\frac{r(\lambda_0)}{d_{m-n+1}\prod_{i\ge3}(\lambda_0-\lambda_i)}
	\end{equation*}
	is a fixed non-zero number. 
	
	Now consider $ \frac1{\lambda-\lambda_0}(r(\lambda)-r(\lambda_0))$. Its value at $ \lambda_0 $ gives 
	\begin{equation*}
		d_{m-n+1}\big((v^*g)_1(g^{-1}v)_1+(v^*g)_2(g^{-1}v)_2\big) \prod_{i\ge3}(\lambda-\lambda_i).
	\end{equation*}
	which implies
	\begin{equation*}
		(v^*g)_1(g^{-1}v)_1+(v^*g)_2(g^{-1}v)_2=e_2:=\frac{r'(\lambda_0)}{d_{m-n+1}\prod_{i\ge3}(\lambda_0-\lambda_i)}
	\end{equation*}
	is a fixed (possibly zero) number. 
	
	In conclusion, the pair $ \big((g^{-1}v)_1,(g^{-1}v)_2\big)\in \bbA^1\times(\bbA^1\backslash\{0\}) $ determines $ (v^*g)_1 $ and $ (v^*g)_2 $. 
	
	Note that the group $ Z_{\GL_2}(\begin{smallmatrix}
		\lambda_0&1\\&\lambda_0
	\end{smallmatrix})=\{(\begin{smallmatrix}
	a&b\\&a
\end{smallmatrix})\} $ acts fully faithfully on the space $ \bbA^1\times(\bbA^1\backslash\{0\}) $. We have the action of $ \GLn $ on $ (x=x_0,v=(0,1,\ldots,1)^\transpose,v^*=(e_1,\ldots,e_n)) $ gives all possibilities of the fiber with trivial stabilizers. 
\end{proof}
\subsection{Localization}
\begin{prop}
	$ \OS $ is normal. 
\end{prop}
\begin{proof}
	Choose a splitting $ \VF/\VO \rightarrow \VF $. We regard sheaves in $ D(\VF/(\VO,\psi)) $ as sheaves in $ D(\VF/\VO) $. Then we have 
	\begin{equation*}
		\OS=\bigoplus_{W\in\Irr\Gl}\Hom(\dV,\dV\cG W)=\bigoplus_{W\in\Irr\Gl}i_0^!(\dV\cG W).
	\end{equation*}
	As direct sums of costalks, $ \OS $ is a free $ H^*_\GLn(\pt)\otimes H^*_\GLm(\pt) $-module. 
\end{proof}
We can use localization to calculate it by fixed points of torus actions. 

For the generic point in $ t\in\cartan_n\times\cartan_m $, the corresponding $ T $-action on $ \Hom(\bbC^n,\bbC^m) $ has fixed point $ \{0\} $. Let $ i_0\colon\{0\}\rightarrow V_F $ be the embedding. Hence we have
\begin{align*}
	&\OS\otimes_{\bbC[\cartan_n/\Sn\times\cartan_m/\Sm]}\bbC(\cartan_n/\Sn\times\cartan_m/\Sm)\\\simeq&\Hom(i_0^*\dV,i_0^*(\dV\cn\calO(\GLn)))\otimes_{\bbC[\cartan_n/\Sn\times\cartan_m/\Sm]}\bbC(\cartan_n/\Sn\times\cartan_m/\Sm)\\=&\Hom(\bbC_\pt,\bbC_\pt\otimes\sh{\calO(\GLn)})\otimes_{\bbC[\cartan_n/\Sn\times\cartan_m/\Sm]}\bbC(\cartan_n/\Sn\times\cartan_m/\Sm)\\=&\bbC(\cartan_n/\Sn\times\cartan_m/\Sm)\otimes\sh{\calO(\GLn)}. 
\end{align*}

For the point $ t\in\cartan_n\times\cartan_m $ lies in the hyperplane given by the resultant, the corresponding $ T $-action on $ \Hom(\bbC^n,\bbC^m) $ has fixed point $ \Hom(\bbC,\bbC) $, where $ \bbC\subset\bbC^n $ and $ \bbC\subset\bbC^m $ are the eigenspaces. Let $ i_2 $ be the embedding. 

To calculate $ i_2^*(\dV\cn\calO(\GLn)) $, one can first pull back through $ i_1\colon\Hom(\bbC^n,\bbC)\rightarrow V $. Then we have $ i_1^*(\dV\cn\calO(\GLn))=\delta_{(\bbC^n)^*}\!\cn\calO(\GLn)\simeq \delta_{(\bbC^n)^*}\!\conv\limits_{\GL_1}\!\gRes\calO(\GLn)=\oplus_{k\in\bbZ}\IC_k\otimes \calO(\GLn)_k^\shear $. Here the isomorphism is given in \cite{Ly11}. 

By calculation, we have $ \Hom(i_2^*\IC_0,i_2^*\IC_k)=\begin{cases}
	\bbC[-kn]&k\ge0\\\bbC[2k-kn]&k\le0
\end{cases}$, and hence
\begin{equation*}
	\OS_t=\Hom(i_2^*\dV,i_2^*(\dV\cn\calO(\GLn)))=\oplus_{k\in\bbZ}\calO(\GLn)_k^\shear=\sh{\calO(\GLn\times+/\Gm)}. 
\end{equation*}
\begin{cor}
	$ \OS $ is commutative. 
\end{cor}
In conclusion, the algebraic map $ \calO(S_f^\shear)\rightarrow \OS $ is an isomorphism over $ \Spec H^*_\GLn(\pt)\otimes H^*_\GLm(\pt) $ up to a codimension $ 2 $ subspace. By localization theorem, these two algebras are isomorphic. 

\end{document}